\documentclass{amsart}
\usepackage{amssymb,amsmath,amsthm}
\newtheorem{theorem}{Theorem}[section]
\newtheorem{lemma}[theorem]{Lemma}
\newtheorem{remark}{Remark}[section]
\usepackage{esint}
\usepackage{xcolor}

\usepackage{hyperref}
\hypersetup{
	colorlinks = true,%
	citecolor = blue,
	filecolor=red,%
	linkcolor = [rgb]{0.65,0.0,0.0},%
	anchorcolor = red,
	pagecolor = red,
	urlcolor= [rgb]{0.65,0.0,0.0}
	linktocpage=true,
	pdfpagelabels=true,
	bookmarksnumbered=true,
}

\numberwithin{equation}{section}

\usepackage{color}

\begin{document}

\title{Fractional Elliptic problem with Finite many critical Hardy--Sobolev Exponents}




\author{Yu Su}
\address{School of Mathematics and Statistics, Central South University, Changsha, 410083 Hunan, P.R.China.}
\curraddr{}
\email{yizai52@qq.com}
\thanks{}

\author{Haibo Chen}
\address{School of Mathematics and Statistics, Central South University, Changsha, 410083 Hunan, P.R.China.}
\curraddr{}
\email{math\_chb@163.com}
\thanks{This research was supported by National Natural Science Foundation of China 11671403.}



\date{}

\dedicatory{}

\begin{abstract}
In this paper,
we  consider the following problem:
$$
(-\Delta)^{s} u
-\frac{\zeta u}{|x|^{2s}}
=
\sum_{i=1}^{k}
\frac{|u|^{2^{*}_{s,\theta_{i}}-2}u}
{|x|^{\theta_{i}}}
,
\mathrm{~in~}
\mathbb{R}^{N},
$$
where
$N\geqslant3$,
$s\in(0,1)$,
$\zeta\in
\left[
0,4^{s}\frac{\Gamma(\frac{N+2s}{4})}{\Gamma(\frac{N-2s}{4})}
\right)$,
$2^{*}_{s,\theta_{i}}=\frac{2(N-\theta_{i})}{N-2s}$
are the critical Hardy--Sobolev  exponents,
the parameters
$\theta_{i}$
satisfy a suitable assumption.
By using Morrey space,
refinement of Hardy--Sobolev inequality
and variational method,
we
establish the existence of nonnegative solution.
Our result generalizes the result obtained by
Chen [Electronic J.  Differ. Eq.
(2018)
1--12 \cite{Chen2018}].
\par
\end{abstract}

\maketitle

\section{Introduction}
In this paper,
we  consider the following problem:
$$
(-\Delta)^{s} u
-\frac{\zeta u}{|x|^{2s}}
=
\sum_{i=1}^{k}
\frac{|u|^{2^{*}_{s,\theta_{i}}-2}u}
{|x|^{\theta_{i}}}
,
\mathrm{~in~}
\mathbb{R}^{N},
\eqno(\mathcal{P})
$$
where
$N\geqslant3$,
$s\in(0,1)$,
$\zeta\in
\left[
0,4^{s}\frac{\Gamma(\frac{N+2s}{4})}{\Gamma(\frac{N-2s}{4})}
\right)$,
$2^{*}_{s,\theta_{i}}=\frac{2(N-\theta_{i})}{N-2s}$
are the critical Hardy--Sobolev  exponents,
the parameters
$\theta_{i}$
satisfy the assumption:

\noindent
($H_{1}$)
$0<\theta_{1}<\cdots<\theta_{k}<2s$
($k\in \mathbb{N},~2< k<\infty$),
and
$2\theta_{k}-\theta_{1}\in(0,2s)$.

The fractional Laplacian $(-\Delta)^{s}$
of a function
$u: \mathbb{R}^{N}\rightarrow\mathbb{R}$
can be defined as
$$(-\Delta)^{s}u=\mathcal{F}^{-1}(|\xi|^{2s}\mathcal{F}(u)(\xi)),~\mathrm{for~all}~\xi\in \mathbb{R}^{N},$$
and for $u\in C^{\infty}_{0}(\mathbb{R}^{N})$,
where $\mathcal{F}(u)$ denotes the Fourier transform of $u$.
The operator
$(-\Delta)^{s}$
in
$\mathbb{R}^{N}$
is a nonlocal pseudo--differential operator taking the form
\begin{equation*}
\begin{aligned}
(-\Delta)^{s}u(x)
=
C_{N,s}\mathrm{P. V.}
\int_{\mathbb{R}^{N}}
\frac{u(x)-u(y)}{|x-y|^{N+2s}}
\mathrm{d}y,
\end{aligned}
\end{equation*}
where
$\mathrm{P. V.}$
is the Cauchy principal value and
$C_{N,s}$
is a normalization constant.
The fractional
power of Laplacian is the infinitesimal generator of L\'{e}vy stable diffusion process and arise
in anomalous diffusion in plasma, population dynamics, geophysical fluid dynamics, flames
propagation, minimal surfaces and game theory
(see
\cite{Applebaum2004,Caffarelli2012,Garroni2005}).

In previous twenty years, the nonlocal elliptic problems have been investigated by many researchers, for example,
\cite{BisciBook,Servadei2012,Servadei2013,Servadei2013DCDS}
for subcritical case (Sobolev),
\cite{O.Alves2016,Chen2015ZAMP,Bisci2016AA,Bisci2015,Servadei2015,Servadei2013CPAA,Yang2015DJDE,Dipierro2016CV}
for critical sobolev case,
\cite{d'Avenia2015,Mukherjee2017Fractional,Yang2017ANS,ZhangBL2017}
for critical Hardy--Littlewood--Sobolev case.
Moreover, a great attention has been devoted to
study the existence of solutions for the nonlocal problems with
critical Hardy--Sobolev term.
Yang \cite{Yang2015NA}
studied the following minimizing problem:
\begin{equation}\label{1}
\begin{aligned}
H_{0,\theta}:=
\inf_{u\in D^{s,2}(\mathbb{R}^{N})\setminus\{0\}}
\frac{
\int_{\mathbb{R}^{N}}
 |(-\Delta)^{\frac{s}{2}}u(x)|^{2}
\mathrm{d}x}
{\left(
\int_{\mathbb{R}^{N}}
\frac{|u|^{2^{*}_{s,\theta}}}{|x|^{\theta}}
\mathrm{d}x
\right)^{\frac{2}{2^{*}_{s,\theta}}}},
\end{aligned}
\end{equation}
where $s\in(0,\frac{N}{2})$
and
$\theta\in(0,2s)$.
By using
Morrey space,
refinement of Hardy--Sobolev inequality
and variational method,
he showed that
$H_{0,\theta}$
is achieved by a positive, radially symmetric and strictly decreasing function.

Ghoussoub and Shakerian \cite{Ghoussoub2016} investigated the following minimizing problem:

\begin{equation}\label{2}
\begin{aligned}
H_{\zeta,\theta}:=
\inf_{u\in D^{s,2}(\mathbb{R}^{N})\setminus\{0\}}
\frac{
\int_{\mathbb{R}^{N}}
 |(-\Delta)^{\frac{s}{2}}u(x)|^{2}
\mathrm{d}x
-
\zeta
\int_{\mathbb{R}^{N}}
\frac{ |u|^{2}}{|x|^{2s}}
\mathrm{d}x}
{\left(
\int_{\mathbb{R}^{N}}
\frac{|u|^{2^{*}_{s,\theta}}}{|x|^{\theta}}
\mathrm{d}x
\right)^{\frac{2}{2^{*}_{s,\theta}}}},
\end{aligned}
\end{equation}
where $s\in(0,1)$,
$N\geqslant2s$,
$\theta\in(0,2s)$
and
$\zeta\in
\left(
-\infty,4^{s}\frac{\Gamma(\frac{N+2s}{4})}{\Gamma(\frac{N-2s}{4})}
\right)$.
Applying Ekeland's variational principle,
for
$s\in(0,1)$,
$\theta\in(0,2s)$
and
$\zeta\in
\left[
0,4^{s}\frac{\Gamma(\frac{N+2s}{4})}{\Gamma(\frac{N-2s}{4})}
\right)$,
they showed that
any non--negative minimizer for
$H_{\zeta,\theta}$
is positive, radially symmetric and radially decreasing.
Furthermore,
they also considered the following problem in \cite{Ghoussoub2016},
\begin{equation}\label{3}
\begin{aligned}
(-\Delta)^{s} u
-\zeta\frac{u}{|x|^{2s}}
=
\frac{|u|^{2^{*}_{s,\theta}-2}u}{|x|^{\theta}}
+
|u|^{2^{*}_{s}-2}u,
\mathrm{~in~}
\mathbb{R}^{N},
\end{aligned}
\end{equation}
where $s\in(0,1)$,
$N\geqslant2s$,
$\theta\in(0,2s)$,
$\zeta\in
\left[
0,4^{s}\frac{\Gamma(\frac{N+2s}{4})}{\Gamma(\frac{N-2s}{4})}
\right)$,
and
$2^{*}_{s}=\frac{2N}{N-2s}$
is the critical Sobolev  exponent.
They combined the s--harmonic extension with the concentration compactness principle
to investigate
the existence of solutions for problem \eqref{3}.

Chen \cite{Chen2018} extended the study of problem \eqref{3} to the following problem:
\begin{equation}\label{4}
\begin{aligned}
(-\Delta)^{s} u
-\zeta\frac{ u}{|x|^{2s}}
=
\frac{|u|^{2^{*}_{s,\theta_{1}}-2}u}{|x|^{\theta_{1}}}
+
\frac{|u|^{2^{*}_{s,\theta_{2}}-2}u}{|x|^{\theta_{2}}},
\mathrm{~in~}
\mathbb{R}^{N},
\end{aligned}
\end{equation}
where
$N\geqslant3$,
$s\in(0,1)$,
$\zeta\in
\left[
0,4^{s}\frac{\Gamma(\frac{N+2s}{4})}{\Gamma(\frac{N-2s}{4})}
\right)$,
$\theta_{1},\theta_{2}\in(0,2s)$
and
$\theta_{1}\not=\theta_{2}$.
By using concentration compactness principle and mountain pass lemma,
they obtain the existence of positive solutions to problem
\eqref{3}.

It is worth pointing out that
there are many other kinds of problem involving two critical nonlinearities,
such as the Laplacian
$-\Delta$
(see \cite{Li2012,Seok2018,Zhong2016}),
the p--Laplacian
$-\Delta_{p}$
(see \cite{Pucci2009}),
the biharmonic operator
$\Delta^{2}$ (see \cite{Bhakta2015}),
and the fractional operator
$(-\Delta)^{s}$
(see \cite{Yang2017ANS,ZhangBL2017}).

{\bf A nature and interesting question arises:
can we extend the study of problem
(\ref{4})
in the finite many critical Hardy--Sobolev exponents?}

We answer above question in this paper.
By using the
refinement of Hardy--Sobolev inequality,
Morrey space
and Mountain Pass Theorem,
we establish the existence of  nontrivial weak solutions of problem
$(\mathcal{P})$.
The main
result of this paper is as follows.
\begin{theorem}\label{theorem1}
Let
$N\geqslant3$,
$s\in(0,1)$,
$\zeta\in
\left[
0,4^{s}\frac{\Gamma(\frac{N+2s}{4})}{\Gamma(\frac{N-2s}{4})}
\right)$
and
$( H_{1})$
hold.
Then
problem
$(\mathcal{P})$
has a nonnegative solution.
\end{theorem}
\begin{remark}
Problem
$(\mathcal{P})$ is invariant under the weighted dilation
$$u\mapsto \tau^{\frac{N-2s}{2}}u(\tau x).$$
Therefore, it is well known that the mountain pass theorem does not yield critical points,
but only the Palais--Smale sequences.
In this type of situation,
it is necessary to show the non--vanishing of Palais--Smale sequences.
There are finite  many critical Hardy--Sobolev  exponents in problem
$(\mathcal{P})$,
it is difficult to show the non--vanishing of Palais--Smale sequences.
In order overcome this difficult,
we establish two new inequalities in Lemma \ref{lemma3} and Lemma \ref{lemma4}.
By using the inequaliteis,
we show the non--vanishing of Palais--Smale sequences in Lemma \ref{lemma8}.
\end{remark}
\begin{remark}
The loss of compactness due to the critical Hardy--Sobolev
exponent which makes it difficult to verify the $(PS)_{c}$ condition,
where
$0<c<c^{*}$
in Lemma \ref{lemma7}.

In \cite{Chen2018},
by using concentration compactness principle,
the author verified that the energy functional associated with problem \eqref{4} satisfied $(PS)_{c}$ condition.
However,
there are finite many  critical Hardy--Sobolev  exponents.
Therefore, her method is not available.
We overcome this difficult by
the refinement of Hardy--Sobolev inequality
and
Morrey space.
\end{remark}
\section{Preliminaries}
Recall that the space
$H^{s}(\mathbb{R}^{N})$
is defined as
$$H^{s}(\mathbb{R}^{N})=\{u\in L^{2}(\mathbb{R}^{N})
|(-\Delta)^{\frac{s}{2}}u\in L^{2}(\mathbb{R}^{N})\}.$$
This space is endowed with the norm
$$\|u\|_{H}^{2}=\|(-\Delta)^{\frac{s}{2}}u\|_{2}^{2}+\|u\|_{2}^{2}.$$
The space
$D^{s,2}(\mathbb{R}^{N})$
is the completion of
$C^{\infty}_{0}(\mathbb{R}^{N})$
with respect to the norm
$$
\|u\|_{D}^{2}=\|(-\Delta)^{\frac{s}{2}}u\|_{2}^{2}.
$$
It is well known that
$\Lambda=4^{s}\frac{\Gamma^{2}(\frac{N+2s}{4})}{\Gamma^{2}(\frac{N-2s}{4})}$
is the best constant in the Hardy inequality
$$
\Lambda
\int_{\mathbb{R}^{N}}
\frac{ |u|^{2}}{|x|^{2s}}
\mathrm{d}x
\leqslant
\|u\|_{D}^{2}
,~~
\mathrm{for~any~}
u\in
D^{s,2}(\mathbb{R}^{N}).
$$
By Hardy inequality and $\zeta\in[0,\Lambda)$,
we derive that
$$
\|u\|_{\zeta}^{2}
=
\|u\|_{D}^{2}
-
\zeta
\int_{\mathbb{R}^{N}}
\frac{ |u|^{2}}{|x|^{2s}}
\mathrm{d}x,
$$
is an equivalent norm in
$D^{s,2}(\mathbb{R}^{N})$,
and the following inequalities hold:
$$\left(1-\frac{\zeta}{\Lambda}\right)\|u\|_{D}^{2}\leqslant\|u\|_{\zeta}^{2}\leqslant\|u\|_{D}^{2}.$$
For
$s\in(0,1)$
and
$\theta\in(0,2s)$,
we define the best constant:
\begin{equation}\label{10}
\begin{aligned}
S_{s}:=
\inf_{u\in D^{s,2}(\mathbb{R}^{N})\setminus\{0\}}
\frac{\|u\|_{D}^{2}
}
{\left(
\int_{\mathbb{R}^{N}}
|u|^{2^{*}_{s}}
\mathrm{d}x
\right)^{\frac{2}{2^{*}_{s}}}},
\end{aligned}
\end{equation}
where
$S_{s}$
is attained in
$\mathbb{R}^{N}$.
For
$s\in(0,1)$,
$\theta\in(0,2s)$
and
$\zeta\in[0,\Lambda)$,
we define the best constant:
\begin{equation}\label{11}
\begin{aligned}
H_{\zeta,\theta}:=
\inf_{u\in D^{s,2}(\mathbb{R}^{N})\setminus\{0\}}
\frac{\|u\|_{D}^{2}
-
\zeta
\int_{\mathbb{R}^{N}}
\frac{ |u|^{2}}{|x|^{2s}}
\mathrm{d}x}
{\left(
\int_{\mathbb{R}^{N}}
\frac{|u|^{2^{*}_{s,\theta}}}{|x|^{\theta}}
\mathrm{d}x
\right)^{\frac{2}{2^{*}_{s,\theta}}}}
\end{aligned}
\end{equation}
where
$H_{\zeta,\theta}$
is attained in
$\mathbb{R}^{N}$
(see \cite{Ghoussoub2016}).
A measurable function
$u:\mathbb{R}^{N}\rightarrow \mathbb{R}$
belongs to the Morrey space
$\|u\|_{\mathcal{L}^{p,\varpi}}(\mathbb{R}^{N})$
with
$p\in[1,\infty)$
and
$\varpi\in(0,N]$
if and only if
$$
\|u\|^{p}_{\mathcal{L}^{p,\varpi}(\mathbb{R}^{N})}
=
\sup_{R>0,x\in\mathbb{R}^{N}}
R^{\varpi-N}
\int_{B(x,R)}
|u(y)|^{p}
\mathrm{d}y
<\infty.
$$
\begin{lemma}
\label{lemma1}
$\left.\right.$
\cite[Theorem 1]{Palatucci2014}
For
$s\in (0,\frac{N}{2})$,
there exists
$C_{1}>0$
such that
for
$\iota$
and
$\vartheta$
satisfying
$\frac{2}{2^{*}_{s}}\leqslant\iota<1$,
$1\leqslant \vartheta<2^{*}_{s}=\frac{2N}{N-2s}$,
we have
\begin{align*}
\left(
\int_{\mathbb{R}^{N}}
|u|^{2^{*}_{s}}
\mathrm{d}x
\right)^{\frac{1}{2^{*}_{s}}}
\leqslant
C_{1}
\|u\|_{D}^{\iota}
\|u\|_{\mathcal{L}^{\vartheta,\frac{\vartheta(N-2s)}{2}}(\mathbb{R}^{N})}^{1-\iota},
\end{align*}
for any
$u\in D^{s,2}(\mathbb{R}^{N})$.
\end{lemma}
We introduce the energy functional associated to  problem $(\mathcal{P})$ by
\begin{equation*}
\begin{aligned}
I(u)
=&
\frac{1}{2}
\|u\|_{\zeta}^{2}
-
\sum_{i=1}^{k}
\frac{1}{2^{*}_{s,\theta_{i}}}
\int_{\mathbb{R}^{N}}
\frac{|u|^{2^{*}_{s,\theta_{i}}}}
{|x|^{\theta_{i}}}
\mathrm{d}x.
\end{aligned}
\end{equation*}
The Nehari manifold associated with
problem $(\mathcal{P})$,
which is  defined by
$$\mathcal{N}=\{u\in D^{s,2}(\mathbb{R}^{N})|\langle I^{'}(u),u\rangle=0,~u\not=0 \},$$
and
$$
c_{0}
=\inf_{u\in\mathcal{N}}
I(u),
~
c_{1}
=\inf_{u\in D^{s,2}(\mathbb{R}^{N})}\max_{t\geqslant0}
I(tu)
~\mathrm{and}~
c
=
\inf_{\Upsilon\in\Gamma}
\max_{t\in [0,1]}
I(\Upsilon(t)),$$
where
$
\Gamma=
\{
\Upsilon\in C([0,1],D^{s,2}(\mathbb{R}^{N}))
:
\Upsilon(0)=0,
I(\Upsilon(1))<0
\}
$.
\section{Some key inequlities}
In this section,
we show some key inequalities.

In \cite{Yang2015NA,Yang2017ANS},
the authors obtained the refinement of Hardy--Sobolev inequality.
However,
their parameter $\tilde{\vartheta}$
satisfying (see \cite[Theorem 1]{Yang2015NA})
$$1\leqslant \tilde{\vartheta}<2^{*}_{s,\theta}.$$
It is easy to see that
$$
2^{*}_{s,\theta}
=\frac{2(N-\theta)}{N-2s}
<
\frac{2N}{N-2s}
=
2^{*}_{s},$$
for
$s\in (0,\frac{N}{2})$
and
$\theta\in(0,2s)$.
{\bf It is natural to ask the case of
$\tilde{\vartheta}\in[2^{*}_{s,\theta},2^{*}_{s})$}.
In next lemma,
we extend the parameter
$\tilde{\vartheta}$
from
$[1,2^{*}_{s,\theta})$
to
$[1,2^{*}_{s})$
.
\begin{lemma}\label{lemma2}
$\left.\right.$
[Refinement of Hardy--Sobolev inequality]
For
$s\in (0,\frac{N}{2})$
and
$\theta\in(0,2s)$,
there exists
$C_{2}>0$
such that
for
$\iota$
and
$\vartheta$
satisfying
$\frac{2}{2^{*}_{s}}\leqslant\iota<1$,
\textcolor{red}{$1\leqslant \vartheta<2^{*}_{s}$},
we have
\begin{align*}
\left(
\int_{\mathbb{R}^{N}}
\frac{|u|^{2^{*}_{s,\theta}}}{|x|^{\theta}}
\mathrm{d}x
\right)^{\frac{1}{2^{*}_{s,\theta}}}
\leqslant
C_{2}
\|u\|_{D}
^{\frac{\theta(N-2s)+\iota N(2s-\theta)}{2s(N-\theta)}}
\|u\|_{\mathcal{L}^{\vartheta,\frac{\vartheta(N-2s)}{2}}(\mathbb{R}^{N})}^{\frac{N(1-\iota)(2s-\theta)}{2s(N-\theta)}},
\end{align*}
for any
$u\in D^{s,2}(\mathbb{R}^{N})$.
\end{lemma}
\begin{proof}
By using H\"{o}lder inequality and fractional Hardy inequality,
we obtain
\begin{equation}\label{12}
\begin{aligned}
\int_{\mathbb{R}^{N}}
\frac{|u|^{2^{*}_{s,\theta}}}
{|x|^{\theta}}
\mathrm{d}x
=&
\int_{\mathbb{R}^{N}}
\frac{|u|^{\frac{\theta}{s}}}
{|x|^{\theta}}
\cdot
|u|^{\frac{2(N-\theta)}{N-2s}-\frac{\theta}{s}}
\mathrm{d}x\\
\leqslant&
\left(
\int_{\mathbb{R}^{N}}
\frac{|u|^{\frac{\theta}{s}\cdot\frac{2s}{\theta}}}
{|x|^{\theta\cdot\frac{2s}{\theta}}}
\mathrm{d}x
\right)
^{\frac{\theta}{2s}}
\left(
\int_{\mathbb{R}^{N}}
|u|^{\frac{(2s-\theta)N}{(N-2s)s}\cdot\frac{2s}{2s-\theta}}
\mathrm{d}x
\right)
^{1-\frac{\theta}{2s}}
\\
=&
\left(
\int_{\mathbb{R}^{N}}
\frac{|u|^{2}}
{|x|^{2s}}
\mathrm{d}x
\right)
^{\frac{\theta}{2s}}
\left(
\int_{\mathbb{R}^{N}}
|u|^{\frac{2N}{N-2s}}
\mathrm{d}x
\right)
^{\frac{2s-\theta}{2s}}
\\
\leqslant&
\left(
\frac{1}{\Lambda}
\right)
^{\frac{\theta}{2s}}
\|u\|_{D}
^{\frac{\theta}{s}}
\|u\|_{L^{2^{*}_{s}}(\mathbb{R}^{N})}^{\frac{N(2s-\theta)}{s(N-2s)}}.
\end{aligned}
\end{equation}
Combining (\ref{12}) and Lemma \ref{lemma1},
we have
\begin{equation*}
\begin{aligned}
\left(
\int_{\mathbb{R}^{N}}
\frac{|u|^{2^{*}_{s,\theta}}}{|x|^{\theta}}
\mathrm{d}x
\right)^{\frac{1}{2^{*}_{s,\theta}}}
\leqslant&
\left(
\frac{1}{\Lambda}
\right)
^{\frac{\theta(N-2s)}{4s(N-\theta)}}
\|u\|_{D}
^{\frac{\theta(N-2s)}{2s(N-\theta)}}
\|u\|_{L^{2^{*}_{s}}(\mathbb{R}^{N})}^{\frac{N(2s-\theta)}{2s(N-\theta)}}\\
\leqslant&
\left(
\frac{1}{\Lambda}
\right)
^{\frac{\theta(N-2s)}{4s(N-\theta)}}
\|u\|_{D}
^{\frac{\theta(N-2s)}{2s(N-\theta)}}
\left(
C_{1}
\|u\|_{D}^{\iota}
\|u\|_{\mathcal{L}^{\vartheta,\frac{\vartheta(N-2s)}{2}}}^{1-\iota}
\right)^{\frac{N(2s-\theta)}{2s(N-\theta)}}\\
=&
C_{2}
\|u\|_{D}
^{\frac{\theta(N-2s)+\iota N(2s-\theta)}{2s(N-\theta)}}
\|u\|_{\mathcal{L}^{\vartheta,\frac{\vartheta(N-2s)}{2}}(\mathbb{R}^{N})}
^{\frac{N(1-\iota)(2s-\theta)}{2s(N-\theta)}}.
\end{aligned}
\end{equation*}
\end{proof}

\begin{lemma}\label{lemma3}
Let
$s\in(0,\frac{N}{2})$
and
$0<\theta<\tilde{\theta}<2s$.
Then the inequality
\begin{equation*}
\begin{aligned}
\int_{\mathbb{R}^{N}}
\frac{|u|^{2^{*}_{s,\theta}}}
{|x|^{\theta}}
\mathrm{d}x
\leqslant
\left(
\int_{\mathbb{R}^{N}}
\frac{|u|^{2^{*}_{s,\tilde{\theta}}}}
{|x|^{\tilde{\theta}}}
\mathrm{d}x
\right)
^{\frac{\theta}{\tilde{\theta}}}
\left(
\int_{\mathbb{R}^{N}}
|u|^{2^{*}_{s}}
\mathrm{d}x
\right)
^{\frac{\tilde{\theta}-\theta}{\tilde{\theta}}},
\end{aligned}
\end{equation*}
holds for all
$u\in D^{s,2}(\mathbb{R}^{N})$.
\end{lemma}
\noindent
{\bf Proof.}
For any
$u\in D^{s,2}(\mathbb{R}^{N})$.
By using H\"{o}lder inequality and $0<\theta<\tilde{\theta}<2s$,
we obtain
\begin{equation*}
\begin{aligned}
\int_{\mathbb{R}^{N}}
\frac{|u|^{2^{*}_{s,\theta}}}
{|x|^{\theta}}
\mathrm{d}x
=&
\int_{\mathbb{R}^{N}}
\frac{|u|^{\frac{\theta}{\tilde{\theta}}\cdot\frac{2(N-\tilde{\theta})}{N-2s}}}
{|x|^{\theta}}
\cdot
|u|^{\frac{2N}{N-2s}\cdot\frac{\tilde{\theta}-\theta}{\tilde{\theta}}}
\mathrm{d}x\\
\leqslant&
\left(
\int_{\mathbb{R}^{N}}
\frac{|u|^{\frac{\theta}{\tilde{\theta}}\cdot\frac{2(N-\tilde{\theta})}{N-2s}\cdot\frac{\tilde{\theta}}{\theta}}}
{|x|^{\theta\cdot\frac{\tilde{\theta}}{\theta}}}
\mathrm{d}x
\right)
^{\frac{\theta}{\tilde{\theta}}}
\left(
\int_{\mathbb{R}^{N}}
|u|^{\frac{2N}{N-2s}\cdot\frac{\tilde{\theta}-\theta}{\tilde{\theta}}\cdot\frac{\tilde{\theta}}{\tilde{\theta}-\theta}}
\mathrm{d}x
\right)
^{1-\frac{\theta}{\tilde{\theta}}}
\\
=&
\left(
\int_{\mathbb{R}^{N}}
\frac{|u|^{\frac{2(N-\tilde{\theta})}{N-2s}}}
{|x|^{\tilde{\theta}}}
\mathrm{d}x
\right)
^{\frac{\theta}{\tilde{\theta}}}
\left(
\int_{\mathbb{R}^{N}}
|u|^{\frac{2N}{N-2s}}
\mathrm{d}x
\right)
^{\frac{\tilde{\theta}-\theta}{\tilde{\theta}}}.
\end{aligned}
\end{equation*}
\qed

\begin{lemma}\label{lemma4}
Let
$s\in(0,\frac{N}{2})$,
$0<\bar{\theta}<\theta<2s$
and
$2\theta-\bar{\theta}<2s$.
Then the inequality
\begin{equation*}
\begin{aligned}
\int_{\mathbb{R}^{N}}
\frac{|u|^{2^{*}_{s,\theta}}}
{|x|^{\theta}}
\mathrm{d}x
\leqslant
\left(
\int_{\mathbb{R}^{N}}
\frac{|u|^{2^{*}_{s,\bar{\theta}}}}
{|x|^{\bar{\theta}}}
\mathrm{d}x
\right)
^{\frac{1}{2}}
\left(
\int_{\mathbb{R}^{N}}
\frac{|u|^{2^{*}_{s,2\theta-\bar{\theta}}}}
{|x|^{2\theta-\bar{\theta}}}
\mathrm{d}x
\right)
^{\frac{1}{2}},
\end{aligned}
\end{equation*}
holds for all
$u\in D^{s,2}(\mathbb{R}^{N})$.
\end{lemma}
\noindent
{\bf Proof.}
For any
$u\in D^{s,2}(\mathbb{R}^{N})$.
By using H\"{o}lder inequality and $0<\bar{\theta}<\theta<2s$,
we obtain
\begin{equation*}
\begin{aligned}
\int_{\mathbb{R}^{N}}
\frac{|u|^{2^{*}_{s,\theta}}}
{|x|^{\theta}}
\mathrm{d}x
=&
\int_{\mathbb{R}^{N}}
\frac{|u|^{\frac{N-\bar{\theta}}{N-2s}}}
{|x|^{\frac{\bar{\theta}}{2}}}
\cdot
\frac{|u|^{\frac{N-(2\theta-\bar{\theta})}{N-2s}}}
{|x|^{\theta-\frac{\bar{\theta}}{2}}}
\mathrm{d}x\\
\leqslant&
\left(
\int_{\mathbb{R}^{N}}
\frac{|u|^{\frac{2(N-\bar{\theta})}{N-2s}}}
{|x|^{\bar{\theta}}}
\mathrm{d}x
\right)
^{\frac{1}{2}}
\left(
\int_{\mathbb{R}^{N}}
\frac{|u|^{\frac{2[N-(2\theta-\bar{\theta})]}{N-2s}}}
{|x|^{2\theta-\bar{\theta}}}
\mathrm{d}x
\right)
^{\frac{1}{2}}.
\end{aligned}
\end{equation*}
Since
$0<2\theta-\bar{\theta}<2s$,
we get
\begin{equation*}
\begin{aligned}
\int_{\mathbb{R}^{N}}
\frac{|u|^{2^{*}_{s,\theta}}}
{|x|^{\theta}}
\mathrm{d}x
\leqslant
\left(
\int_{\mathbb{R}^{N}}
\frac{|u|^{2^{*}_{s,\bar{\theta}}}}
{|x|^{\bar{\theta}}}
\mathrm{d}x
\right)
^{\frac{1}{2}}
\left(
\int_{\mathbb{R}^{N}}
\frac{|u|^{2^{*}_{s,2\theta-\bar{\theta}}}}
{|x|^{2\theta-\bar{\theta}}}
\mathrm{d}x
\right)
^{\frac{1}{2}}.
\end{aligned}
\end{equation*}
\qed
\section{The proof of theorem \ref{theorem1}}
In this section, we show the existence of nonnegative solution of problems $(\mathcal{P})$.
We prove some properties of the Nehari manifold associated with
problem $(\mathcal{P})$.
\begin{lemma}\label{lemma5}
Assume that the assumptions of Theorem
\ref{theorem1}
hold.
Then
$$c_{0}=\inf_{u\in\mathcal{N}}I(u)>0.$$
\end{lemma}
\noindent
{\bf Proof.}
We divide our proof into two steps.

\noindent
{\bf Step 1.}
We claim that
any limit point of a sequence in
$\mathcal{N}$
is different from zero.
According to
$\langle I^{'}(u),u\rangle=0$
and
(\ref{11}),
for any $u\in\mathcal{N}$,
we obtain
\begin{equation*}
\begin{aligned}
0
=
\langle I^{'}(u),u\rangle
\geqslant&
\|u\|^{2}_{\zeta}
-
\sum_{i=1}^{k}
\frac{1}{H_{\zeta,\theta_{i}}^{\frac{2^{*}_{s,\theta_{i}} }{2}}}
\|u\|_{\zeta}^{2^{*}_{s,\theta_{i}} }.
\end{aligned}
\end{equation*}
From above expression,
we have
\begin{equation}\label{13}
\begin{aligned}
\|u\|^{2}_{\zeta}
\leqslant
\sum_{i=1}^{k}
\frac{1}{H_{\zeta,\theta_{i}}^{\frac{2^{*}_{s,\theta_{i}} }{2}}}
\|u\|_{\zeta}^{2^{*}_{s,\theta_{i}} }.
\end{aligned}
\end{equation}
Set
$$\kappa
:=
\sum_{i=1}^{k}
\frac{1}{H_{\zeta,\theta_{i}}^{\frac{2^{*}_{s,\theta_{i}} }{2}}}
.$$
Applying
(\ref{11}),
we get
$$
0<\kappa<\infty.$$
From
($H_{1}$),
we know
$$2<2^{*}_{s,\theta_{k}}<\cdots<2^{*}_{s,\theta_{1}}<2^{*}_{s}.$$
Now the proof of Step 1 is divided into two cases:
(i)
$\|u\|_{\zeta}\geqslant1$;
(ii)
$\|u\|_{\zeta}<1$.

\noindent
{\bf Case (i)}$\|u\|_{\zeta}\geqslant1$.
From
(\ref{13}),
we have
\begin{equation*}
\begin{aligned}
\|u\|_{\zeta}^{2}
\leqslant&
\kappa
\|u\|_{\zeta}^{2^{*}_{s,\theta_{1}}},
\end{aligned}
\end{equation*}
which implies that
\begin{equation}\label{14}
\begin{aligned}
\|u\|_{\zeta}
\geqslant
\kappa
^{\frac{1}{2-2^{*}_{s,\theta_{1}}}}.
\end{aligned}
\end{equation}
\noindent
{\bf Case (ii)}$\|u\|_{\zeta}<1$.
From
(\ref{13}),
we know
\begin{equation}\label{15}
\begin{aligned}
\|u\|_{\zeta}
\geqslant
\kappa
^{\frac{1}{2-2^{*}_{s,\theta_{k}} }}.
\end{aligned}
\end{equation}
Combining
(\ref{14})
and
(\ref{15}),
we deduce that
\begin{equation}\label{16}
\begin{aligned}
\|u\|_{\zeta}
\geqslant
\begin{cases}
\kappa
^{\frac{1}{2-2^{*}_{s,\theta_{1}} }},
&\kappa<1,\\
\kappa
^{\frac{1}{2-2^{*}_{s,\theta_{k}} }},
&\kappa\geqslant1.
\end{cases}
\end{aligned}
\end{equation}
Hence,
we know that
any limit point of a sequence in $\mathcal{N}$ is different from zero.

\noindent
{\bf Step 2.}
Now,
we claim that
$I$
is bounded from below on
$\mathcal{N}$.
For any
$u\in \mathcal{N}$,
by using
(\ref{16}),
we get
\begin{equation*}
\begin{aligned}
I(u)
\geqslant&
\left(
\frac{1}{2}
-
\frac{1}{2^{*}_{s,\theta_{k}}}
\right)
\|u\|^{2}_{\zeta}
\geqslant
\begin{cases}
\left(
\frac{1}{2}
-
\frac{1}{2^{*}_{s,\theta_{k}}}
\right)
\kappa
^{\frac{2}{2-2^{*}_{s,\theta_{1}} }},
&\kappa\leqslant1,\\
\left(
\frac{1}{2}
-
\frac{1}{2^{*}_{s,\theta_{k}}}
\right)
\kappa
^{\frac{2}{2-2^{*}_{s,\theta_{k}}}},
&\kappa>1.
\end{cases}
\end{aligned}
\end{equation*}
Therefore,
$I$
is bounded from below on
$\mathcal{N}$,
and
$c_{0}>0$.
\qed
\begin{lemma}\label{lemma6}
Assume that the assumptions of Theorem
\ref{theorem1}
hold.
Then

\noindent
(i)
for each
$u\in D^{s,2}(\mathbb{R}^{N})\setminus\{0\}$,
there exists a unique
$t_{u}>0$
such that
$t_{u}u\in \mathcal{N}$;

\noindent
(ii)
$c_{0}=c_{1}=c>0$.
\end{lemma}
\noindent
{\bf Proof.}
The proof is standard, so we sketch it. Further details can be derived as in the proofs of Theorem 4.1 and 4.2 in
\cite{Willem1996}.
We omit it.
\qed

We show that the functional $I$ satisfies the Mountain--Pass geometry, and
estimate the Mountain--Pass levels.
\begin{lemma}\label{lemma7}
Assume that the assumptions of Theorem
\ref{theorem1}
hold.
Then
there exists a
$(PS)_{c}$
sequence of
$I$
at level
$c$,
where
$$0<c<c^{*}=
\min
\left\{
\frac{2s-\theta_{1}}{2(N-\theta_{1})}
H_{\zeta,\theta_{1}} ^{\frac{N-\theta_{1}}{2s-\theta_{1}}}
,
\ldots
,
\frac{2s-\theta_{k}}{2(N-\theta_{k})}
H_{\zeta,\theta_{k}} ^{\frac{N-\theta_{k}}{2s-\theta_{k}}}
\right\}.$$
\end{lemma}
\noindent
{\bf Proof.}
The proof is standard, so we sketch it. Further details can be derived as in the proofs of Theorem 2 in \cite{Pucci2009}, we omit it.
\qed

The following result implies the non--vanishing of
$(PS)_{c}$
sequence.
\begin{lemma}\label{lemma8}
Assume that the assumptions of Theorem
\ref{theorem1}
hold.
Let
$\{u_{n}\}$
be a
$(PS)_{c}$
sequence of
$I$
with
$c\in(0,c^{*})$,
then
\begin{equation*}
\begin{aligned}
\lim_{n\rightarrow\infty}
\int_{\mathbb{R}^{N}}
\frac{|u_{n}|^{2^{*}_{s,\theta_{i}}}}
{|x|^{\theta_{i}}}
\mathrm{d}x
>0,(i=1,\ldots,k).
\end{aligned}
\end{equation*}
\end{lemma}
\noindent
{\bf Proof.}
It is easy to see that
$\{u_{n}\}$
is uniformly bounded in
$D^{s,2}(\mathbb{R}^{N})$.
The proof of this Lemma is divided into three cases:

\noindent
(1)
$
\lim_{n\rightarrow\infty}
\int_{\mathbb{R}^{N}}
\frac{|u_{n}|^{2^{*}_{s,\theta_{1}}}}
{|x|^{\theta_{1}}}
\mathrm{d}x
>0
$;

\noindent
(2)
$
\lim_{n\rightarrow\infty}
\int_{\mathbb{R}^{N}}
\frac{|u_{n}|^{2^{*}_{s,\theta_{k}}}}
{|x|^{\theta_{k}}}
\mathrm{d}x
>0
$;

\noindent
(3)
$
\lim_{n\rightarrow\infty}
\int_{\mathbb{R}^{N}}
\frac{|u_{n}|^{2^{*}_{s,\theta_{j}}}}
{|x|^{\theta_{j}}}
\mathrm{d}x
>0
$,
$(j=2,\ldots,k-1)$;

\noindent
{\bf Case 1.}
Suppose  on the contrary that
\begin{equation}\label{17}
\begin{aligned}
\lim_{n\rightarrow\infty}
\int_{\mathbb{R}^{N}}
\frac{|u_{n}|^{2^{*}_{s,\theta_{1}}}}
{|x|^{\theta_{1}}}
\mathrm{d}x
=0.
\end{aligned}
\end{equation}
From
($H_{1}$),
we know
\begin{equation}\label{18}
\begin{aligned}
0<2\theta_{2}-\theta_{1}<\cdots<2\theta_{k}-\theta_{1}<2s.
\end{aligned}
\end{equation}
Since
$\{u_{n}\}$
is uniformly bounded in
$D^{s,2}(\mathbb{R}^{N})$,
there exists a constant
$0<C<\infty$
such that
$\|u_{n}\|_{D}\leqslant C$.
Applying
(\ref{18})
and
(\ref{11}),
we obtain
\begin{equation}\label{19}
\begin{aligned}
\lim_{n\rightarrow\infty}
\int_{\mathbb{R}^{N}}
\frac{|u_{n}|^{2^{*}_{s,2\theta_{i}-\theta_{1}}}}
{|x|^{2\theta_{i}-\theta_{1}}}
\mathrm{d}x
\leqslant C,
(i=2,\ldots,k).
\end{aligned}
\end{equation}
According to
Lemma
\ref{lemma4},
(\ref{17})
and
(\ref{19}),
we obtain
\begin{equation}\label{20}
\begin{aligned}
&\lim_{n\rightarrow\infty}
\int_{\mathbb{R}^{N}}
\frac{|u_{n}|^{2^{*}_{s,\theta_{i}}}}
{|x|^{\theta_{i}}}
\mathrm{d}x\\
\leqslant&
\left(
\lim_{n\rightarrow\infty}
\int_{\mathbb{R}^{N}}
\frac{|u_{n}|^{2^{*}_{s,\theta_{1}}}}
{|x|^{\theta_{1}}}
\mathrm{d}x
\right)
^{\frac{1}{2}}
\left(
\lim_{n\rightarrow\infty}
\int_{\mathbb{R}^{N}}
\frac{|u_{n}|^{2^{*}_{s,2\theta_{i}-\theta_{1}}}}
{|x|^{2\theta_{i}-\theta_{1}}}
\mathrm{d}x
\right)
^{\frac{1}{2}}\\
=&0
~(i=2,\ldots,k).
\end{aligned}
\end{equation}
By using
(\ref{17}),
(\ref{20})
and the definition of
$(PS)_{c}$
sequence,
we obtain
$$
c
+
o(1)
=
\frac{1}{2}
\|u_{n}\|_{\zeta}^{2}
,$$
and
$$o(1)
=
\|u_{n}\|_{\zeta}^{2}
.
$$
These yield $c=0$.
This contradicts with $c>0$.

\noindent
{\bf Case 2.}
Suppose  on the contrary that
\begin{equation}\label{21}
\begin{aligned}
\lim_{n\rightarrow\infty}
\int_{\mathbb{R}^{N}}
\frac{|u_{n}|^{2^{*}_{s,\theta_{k}}}}
{|x|^{\theta_{k}}}
\mathrm{d}x
=0.
\end{aligned}
\end{equation}
By using
(\ref{10})
and
$\|u_{n}\|_{D}\leqslant C$,
we have
\begin{equation}\label{22}
\begin{aligned}
\lim_{n\rightarrow\infty}
\int_{\mathbb{R}^{N}}
|u_{n}|^{2^{*}_{s}}
\mathrm{d}x
\leqslant C.
\end{aligned}
\end{equation}
Applying
($H_{1}$),
Lemma
\ref{lemma3},
(\ref{21})
and
(\ref{22}),
we obtain
\begin{equation}\label{23}
\begin{aligned}
&
\lim_{n\rightarrow\infty}
\int_{\mathbb{R}^{N}}
\frac{|u_{n}|^{2^{*}_{s,\theta_{i}}}}
{|x|^{\theta_{i}}}
\mathrm{d}x\\
\leqslant&
\left(
\lim_{n\rightarrow\infty}
\int_{\mathbb{R}^{N}}
\frac{|u_{n}|^{2^{*}_{s,\theta_{k}}}}
{|x|^{\theta_{k}}}
\mathrm{d}x
\right)
^{\frac{\theta_{i}}{\theta_{k}}}
\left(
\lim_{n\rightarrow\infty}
\int_{\mathbb{R}^{N}}
|u_{n}|^{2^{*}_{s}}
\mathrm{d}x
\right)
^{\frac{\theta_{k}-\theta_{i}}{\theta_{k}}}\\
=&0
~(i=1,\ldots,k-1).
\end{aligned}
\end{equation}
By using
(\ref{21}),
(\ref{23})
and the definition of
$(PS)_{c}$
sequence,
similar to Case 1,
we get
$
c=0.
$
This is a contradiction.

\noindent
{\bf Case 3.}
Set
$j\in [2,k-1]$.
Suppose on the contrary that
\begin{equation}\label{24}
\begin{aligned}
\lim_{n\rightarrow\infty}
\int_{\mathbb{R}^{N}}
\frac{|u_{n}|^{2^{*}_{s,\theta_{j}}}}
{|x|^{\theta_{j}}}
\mathrm{d}x
=0.
\end{aligned}
\end{equation}
From
($H_{1}$),
we know
\begin{equation*}
\begin{aligned}
0<2\theta_{j+1}-\theta_{j}<\cdots<2\theta_{k}-\theta_{j}<2\theta_{k}-\theta_{1}<2s.
\end{aligned}
\end{equation*}
Similar to
\eqref{20},
we obtain
\begin{equation*}
\begin{aligned}
\lim_{n\rightarrow\infty}
\int_{\mathbb{R}^{N}}
\frac{|u_{n}|^{2^{*}_{s,\theta_{i}}}}
{|x|^{\theta_{i}}}
\mathrm{d}x
=0
~(i=j+1,\ldots,k).
\end{aligned}
\end{equation*}
Similar to \eqref{23},
we have
\begin{equation*}
\begin{aligned}
\lim_{n\rightarrow\infty}
\int_{\mathbb{R}^{N}}
\frac{|u_{n}|^{2^{*}_{s,\theta_{i}}}}
{|x|^{\theta_{i}}}
\mathrm{d}x
=&0
~(i=1,\ldots,j-1).
\end{aligned}
\end{equation*}
Similar to Case 1,
we get
$
c=0.
$
This is a contradiction.
\qed

\noindent
{\bf The proof of Theorem \ref{theorem1}:}
We divide our proof into five steps.

\noindent
{\bf Step 1.}
Since
$\{u_{n}\}$
is a bounded sequence in
$D^{s,2}(\mathbb{R}^{N})$,
up to a subsequence,
we assume that
\begin{align*}
&
u_{n}\rightharpoonup u,
~
\mathrm{in}
~
D^{s,2}(\mathbb{R}^{N}),
~~
u_{n}\rightarrow u,
~
\mathrm{a.e. ~in}
~
\mathbb{R}^{N},\\
&u_{n}\rightarrow u,
~
\mathrm{in}
~
L^{r}_{loc}(\mathbb{R}^{N})
~
\mathrm{for~all}
~
r\in[1,2^{*}_{s}).
\end{align*}
According to
Lemma
\ref{lemma2},
and
Lemma
\ref{lemma8},
there exists
$C>0$
such that
$$
\|u_{n}\|_{\mathcal{L}^{2,N-2s}(\mathbb{R}^{N})}\geqslant C>0.
$$
On the other hand,
since the sequence is bounded in
$D^{s,2}(\mathbb{R}^{N})$
and
$D^{s,2}(\mathbb{R}^{N})\hookrightarrow L^{2^{*}_{s}}(\mathbb{R}^{N})\hookrightarrow \mathcal{L}^{2,N-2s}(\mathbb{R}^{N})$,
we have
$$
\|u_{n}\|_{\mathcal{L}^{2,N-2s}(\mathbb{R}^{N})}\leqslant C,
$$
for some
$C>0$
independent of $n$.
Hence, there exists a positive constant which we denote again by $C$ such
that for any $n$ we obtain
$$
C
\leqslant
\|u_{n}\|_{\mathcal{L}^{2,N-2s}(\mathbb{R}^{N})}\leqslant C^{-1}.
$$
So we may find
$\sigma_{n} > 0$
and
$x_{n}\in \mathbb{R}^{N}$
such that
$$
\frac{1}{\sigma_{n}^{2s}}
\int_{B(x_{n},\sigma_{n})}
|u_{n}(y)|^{2}
\mathrm{d}y
\geqslant
\|u_{n}\|_{\mathcal{L}^{2,N-2s}(\mathbb{R}^{N})}^{2}
-
\frac{C}{2n}
\geqslant
C_{3}>0.
$$
Let
$\bar{u}_{n}(x)=\sigma_{n}^{\frac{N-2s}{2}}u_{n}(x_{n}+\sigma_{n}x)$.
We may readily verify that
$$\tilde{I}(\bar{u}_{n})
=
I(u_{n})\rightarrow c
~\mathrm{and}~
\tilde{I}^{'}(\bar{u}_{n})
\rightarrow0
~\mathrm{as}~n\rightarrow\infty,$$
where
\begin{equation*}
\begin{aligned}
\tilde{I}(\bar{u}_{n})
=&
\frac{1}{2}
\|\bar{u}_{n}\|_{D}^{2}
-
\frac{1}{2}
\int_{\mathbb{R}^{N}}
\frac{|\bar{u}_{n}|^{2}}
{|x+\frac{x_{n}}{\sigma_{n}}|^{2s}}
\mathrm{d}x
-
\sum_{i=1}^{k}
\frac{1}{2^{*}_{s,\theta_{i}}}
\int_{\mathbb{R}^{N}}
\frac{|\bar{u}_{n}|^{2^{*}_{s,\theta_{i}}}}
{|x+\frac{x_{n}}{\sigma_{n}}|^{\theta_{i}}}
\mathrm{d}x.
\end{aligned}
\end{equation*}
Now,
for all
$\varphi\in D^{s,2}(\mathbb{R}^{N})$,
we obtain
\begin{equation*}
\begin{aligned}
|\langle \tilde{I}^{'}(\bar{u}_{n}),\varphi\rangle|
=&
|\langle I^{'}(u_{n}),\bar{\varphi}\rangle|\\
\leqslant&
\|I^{'}(u_{n})\|_{D^{-1}}
\|\bar{\varphi}\|_{D}\\
=&
o(1)
\|\bar{\varphi}\|_{D},
\end{aligned}
\end{equation*}
where
$\bar{\varphi}=\sigma_{n}^{-\frac{N-2s}{2}}\varphi(\frac{x-x_{n}}{\sigma_{n}})$.
Since
$\|\bar{\varphi}\|_{D}=\|\varphi\|_{D}$,
we get
$$\tilde{I}^{'}(\bar{u}_{n})\rightarrow0~
\mathrm{as}~
n\rightarrow\infty.$$
Thus there exists
$\bar{u}$
such that
\begin{align*}
&
\bar{u}_{n}\rightharpoonup \bar{u},
~
\mathrm{in}
~
D^{s,2}(\mathbb{R}^{N}),
~~
\bar{u}_{n}\rightarrow \bar{u},
~
\mathrm{a.e. ~in}
~
\mathbb{R}^{N},\\
&\bar{u}_{n}\rightarrow \bar{u},
~
\mathrm{in}
~
L^{r}_{loc}(\mathbb{R}^{N})
~
\mathrm{for~all}
~
r\in[1,2^{*}_{s}).
\end{align*}
Then
$$
\int_{B(0,1)}
|\bar{u}_{n}(y)|^{2}
\mathrm{d}y
=
\frac{1}{\sigma_{n}^{2s}}
\int_{B(x_{n},\sigma_{n})}
|u_{n}(y)|^{2}
\mathrm{d}y
\geqslant
C_{3}>0.$$
As a result,
$\bar{u}\not\equiv0$.

\noindent
{\bf Step 2.}
Now,
we claim that
$\{\frac{x_{n}}{\sigma_{n}}\}$
is bounded.
If
$\frac{x_{n}}{\sigma_{n}}\rightarrow\infty$,
then for any
$\varphi\in D^{s,2}(\mathbb{R}^{N})$,
we get
\begin{equation}\label{25}
\begin{aligned}
\lim_{n\rightarrow\infty}
\int_{\mathbb{R}^{N}}
\frac{\bar{u}_{n}\varphi}
{|x+\frac{x_{n}}{\sigma_{n}}|^{2s}}
\mathrm{d}x
=
0
~\mathrm{and}~
\lim_{n\rightarrow\infty}
\int_{\mathbb{R}^{N}}
\frac{|\bar{u}_{n}|^{2^{*}_{s,\theta_{i}}-2}\bar{u}_{n}\varphi}
{|x+\frac{x_{n}}{\sigma_{n}}|^{\theta_{i}}}
\mathrm{d}x
=0.
\end{aligned}
\end{equation}
Since
$\bar{u}_{n}\rightharpoonup \bar{u}$
weakly in
$D^{s,2}(\mathbb{R}^{N})$,
we know
\begin{equation}\label{26}
\begin{aligned}
&\lim_{n\rightarrow\infty}
\int_{\mathbb{R}^{N}}
\int_{\mathbb{R}^{N}}
\frac{(\bar{u}_{n}(x)-\bar{u}_{n}(y))(\varphi(x)-\varphi(y))}
{|x-y|^{N+2s}}
\mathrm{d}x
\mathrm{d}y\\
=&
\int_{\mathbb{R}^{N}}
\int_{\mathbb{R}^{N}}
\frac{(\bar{u}(x)-\bar{u}(y))(\varphi(x)-\varphi(y))}
{|x-y|^{N+2s}}
\mathrm{d}x
\mathrm{d}y.
\end{aligned}
\end{equation}
Applying
$\lim\limits_{n\rightarrow\infty}\langle \tilde{I}^{'}(\bar{u}_{n}),\varphi\rangle\rightarrow0$,
\eqref{25}
and
\eqref{26},
we have
\begin{equation*}
\begin{aligned}
\int_{\mathbb{R}^{N}}
\int_{\mathbb{R}^{N}}
\frac{(\bar{u}(x)-\bar{u}(y))(\varphi(x)-\varphi(y))}
{|x-y|^{N+2s}}
\mathrm{d}x
\mathrm{d}y
=0.
\end{aligned}
\end{equation*}
Let
$\varphi=\bar{u}$.
Then
$\|\bar{u}\|_{D}=0$,
so
$u(x)=u(y)$
a.e.
$(x,y)\in \mathbb{R}^{N}\times \mathbb{R}^{N}$,
that is 
$u=a\in\mathbb{R}$
a.e. in
$\mathbb{R}^{N}$ (see \cite[Page 27, Line 12]{BisciBook}).

If $a\not=0$, then $0<\|\bar{u}\|_{L^{2^{*}_{s}}(\mathbb{R}^{N})}^{2}\leqslant \frac{1}{S_{s}}\|\bar{u}\|_{D}^{2}=0$,
which is a contradiction.

If $a=0$,
then it contradicts with $\bar{u}\not\equiv0$.

Hence, $\{\frac{x_{n}}{\sigma_{n}}\}$ is bounded.

\noindent
{\bf Step 3.}
In this step,
we study another
$(PS)_{c}$
sequence of
$I$.
Let
$\tilde{u}_{n}(x)=\sigma_{n}^{\frac{N-2s}{2}}u_{n}(\sigma_{n}x)$.
Then we can
verify that
$$I(\tilde{u}_{n})=I(u_{n})\rightarrow c,~
I^{'}(\tilde{u}_{n})\rightarrow 0
~\mathrm{as}~n\rightarrow\infty.$$
Arguing as before, we have
\begin{align*}
&\tilde{u}_{n}\rightharpoonup \tilde{u}
~
\mathrm{in}
~
D^{1,2}(\mathbb{R}^{N}),~
\tilde{u}_{n}\rightarrow \tilde{u}
~
\mathrm{a.e. ~in}
~
\mathbb{R}^{N},\\
&\tilde{u}_{n}\rightarrow \tilde{u}
~
\mathrm{in}
~
L^{r}_{loc}(\mathbb{R}^{N})~~\mathrm{for~all~}r\in[1,2^{*}_{s}).
\end{align*}
Since  $\{\frac{x_{n}}{\sigma_{n}}\}$ is bounded,
there exists
$\tilde{R}>0$ such that
$$
\int_{B(0,\tilde{R})}
|\tilde{u}_{n}(y)|^{2}
\mathrm{d}y
>
\int_{B(\frac{x_{n}}{\sigma_{n}},1)}
|\tilde{u}_{n}(y)|^{2}
\mathrm{d}y
=
\frac{1}{\sigma_{n}^{2s}}
\int_{B(x_{n},\sigma_{n})}
|u_{n}(y)|^{2}
\mathrm{d}y
\geqslant
C_{3}>0.$$
As a result,
$\tilde{u}\not\equiv0$.

\noindent
{\bf Step 4.}
In this step,
we show
$\tilde{u}_{n}\rightarrow \tilde{u}$
strongly in
$D^{s,2}(\mathbb{R}^{N})$.
It is easy to see that
\begin{equation}\label{27}
\begin{aligned}
\langle I^{'}(\tilde{u}),\varphi\rangle=0.
\end{aligned}
\end{equation}
Combining
\eqref{27}
and
$\tilde{u}\not\equiv0$,
we get
$\tilde{u}\in \mathcal{N}$.
Set
\begin{align*}
K(u)
=
\sum_{i=1}^{k}
\left(
\frac{1}{2}
-
\frac{1}{2^{*}_{s,\theta_{i}}}
\right)
\int_{\mathbb{R}^{N}}
\frac{|u|^{2^{*}_{s,\theta_{i}}}}
{|x|^{\theta_{i}}}
\mathrm{d}x.
\end{align*}
Applying
Lemma \ref{lemma6},
Br\'{e}zis--Lieb lemma,
$\tilde{u}\in \mathcal{N}$
and
Lemma \ref{lemma5},
we obtain
\begin{equation}\label{28}
\begin{aligned}
c_{0}=
c
=&
I(\tilde{u}_{n})
-
\frac{1}{2}
\langle I^{'}(\tilde{u}_{n}),\tilde{u}_{n}\rangle\\
=&
\lim_{n\rightarrow\infty}K(\tilde{u}_{n})+o(1)\\
\geqslant&
K(\tilde{u})
+o(1)\\
=&
I(\tilde{u})
-
\frac{1}{2}
\langle I^{'}(\tilde{u}),\tilde{u}\rangle
=
I(\tilde{u})
\geqslant
c_{0}.
\end{aligned}
\end{equation}
\textcolor{red}{Therefore, the inequalities above have to be equalities}.
We know
\begin{align*}
\lim\limits_{n\rightarrow\infty}
K(\tilde{u}_{n})
=
K(\tilde{u}).
\end{align*}
By using Br\'{e}zis--Lieb lemma again,
we have
\begin{align*}
\lim\limits_{n\rightarrow\infty}
K(\tilde{u}_{n})
-
\lim\limits_{n\rightarrow\infty}
K(\tilde{u}_{n}-\tilde{u})
=
K(\tilde{u})+o(1).
\end{align*}
Hence,
we deduce that
\begin{align*}
\lim\limits_{n\rightarrow\infty}
K(\tilde{u}_{n}-\tilde{u})
=
0,
\end{align*}
which implies that
\begin{equation}\label{29}
\begin{aligned}
\lim\limits_{n\rightarrow\infty}
\int_{\mathbb{R}^{N}}
\frac{
|\tilde{u}_{n}-\tilde{u}|^{2^{*}_{s,\theta_{i}}}
}
{|x|^{\theta_{i}}}
\mathrm{d}x
=0,~\mathrm{for~all}~
i=1,\ldots,k.
\end{aligned}
\end{equation}
According to
$\langle I^{'}(\tilde{u}_{n}),\tilde{u}_{n}\rangle=o(1)$,
$\langle I^{'}(\tilde{u}),\tilde{u}\rangle=0$
and Br\'{e}zis--Lieb lemma,
we obtain
\begin{equation*}
\begin{aligned}
o(1)=&
\langle I^{'}(\tilde{u}_{n}),\tilde{u}_{n}\rangle
-
\langle I^{'}(\tilde{u}),\tilde{u}\rangle\\
=&
\|\tilde{u}_{n}-\tilde{u}\|_{\zeta}^{2}
-\sum_{i=1}^{k}
\int_{\mathbb{R}^{N}}
\frac{
|\tilde{u}_{n}-\tilde{u}|^{2^{*}_{s,\theta_{i}}}
}
{|x|^{\theta_{i}}}
\mathrm{d}x
+o(1),
\end{aligned}
\end{equation*}
which implies that
\begin{equation}\label{30}
\begin{aligned}
\lim\limits_{n\rightarrow\infty}
\|\tilde{u}_{n}-\tilde{u}\|_{\zeta}^{2}
=
\lim\limits_{n\rightarrow\infty}
\int_{\mathbb{R}^{N}}
\frac{
|\tilde{u}_{n}-\tilde{u}|^{2^{*}_{s,\theta_{i}}}
}
{|x|^{\theta_{i}}}
\mathrm{d}x+o(1).
\end{aligned}
\end{equation}
Combining
(\ref{29})
and
(\ref{30}),
we get
\begin{equation*}
\begin{aligned}
\lim\limits_{n\rightarrow\infty}
\|\tilde{u}_{n}-\tilde{u}\|_{\zeta}^{2}
=
0.
\end{aligned}
\end{equation*}
Since
$\tilde{u}\not\equiv0$,
we know that
$\tilde{u}_{n}\rightarrow \tilde{u}$
strongly in
$D^{s,2}(\mathbb{R}^{N})$.

\noindent
{\bf Step 5.}
By using \eqref{28} again,
we know
$I(\tilde{u})=c$,
which means that $\tilde{u}$ is a nontrivial solution of problem $(\mathcal{P})$ at the
energy level $c$.
We have
\begin{equation*}
\begin{aligned}
0=&
\langle I^{'}(\tilde{u}),\tilde{u}^{-}\rangle\\
=&
\int_{\mathbb{R}^{N}}
\int_{\mathbb{R}^{N}}
\frac{(\tilde{u}(x)-\tilde{u}(y))(\tilde{u}^{-}(x)-\tilde{u}^{-}(y))}{|x-y|^{N+2s}}
\mathrm{d}x
\mathrm{d}y
-
\zeta
\int_{\mathbb{R}^{N}}
\frac{\tilde{u}\tilde{u}^{-}}{|x|^{2s}}
\mathrm{d}x
\\
&
-
\sum_{i=1}^{k}
\int_{\mathbb{R}^{N}}
\frac{|\tilde{u}|^{2^{*}_{s,\theta_{i}}-2}\tilde{u}\tilde{u}^{-}}
{|x|^{\theta_{i}}}
\mathrm{d}x,
\end{aligned}
\end{equation*}
where
$\tilde{u}^{-}=\max\{0,-\tilde{u}\}$.
For a.e.
$x,y\in \mathbb{R}^{N}$,
we obtain
$$(\tilde{u}(x)-\tilde{u}(y))(\tilde{u}^{-}(x)-\tilde{u}^{-}(y))\leqslant -|\tilde{u}^{-}(x)-\tilde{u}^{-}(y)|^{2}.$$
Then,
we get
\begin{equation*}
\begin{aligned}
0\leqslant
-
\|\tilde{u}^{-}\|_{D}^{2}
-
\zeta
\int_{\mathbb{R}^{N}}
\frac{|\tilde{u}^{-}|^{2}}{|x|^{2s}}
\mathrm{d}x
-
\sum_{i=1}^{k}
\int_{\mathbb{R}^{N}}
\frac{|\tilde{u}^{-}|^{2^{*}_{s,\theta_{i}}}}
{|x|^{\theta_{i}}}
\mathrm{d}x
\leqslant
-\|\tilde{u}^{-}\|_{D}^{2}.
\end{aligned}
\end{equation*}
Thus,
$\|\tilde{u}^{-}\|_{D}^{2}=0$.
Hence,
we can choose
$\tilde{u}\geqslant0$.
\qed

\end{document}